\documentclass[reqno,11pt]{amsart}
\usepackage[T1]{fontenc}
\usepackage[cp1250]{inputenc}
\usepackage{epsfig}
\newcommand\N{\mathbf{N}}
\newcommand\R{\mathbf{R}}

\newcommand{\eps}{\varepsilon}
\newcommand{\HH}{{\mathcal H}}
\renewcommand{\H}{\HH^1}

\newcommand{\bib}[4]{\bibitem{#1}{\sc#2: }{\it#3. }{#4.}}

\newcommand{\forget}[1]{}
\def\Om{{\Omega}}  
\def\om2{{\Om\times\Om}}

\def\F{{\mathcal F}}

\def\diam{\mathrm{diam}\,}
\def\div{\mathrm{div}\,}
\def\dist{\mathrm{dist}\,}
\def\eps{\varepsilon}

\newcommand{\Ll}{\mathcal{L}}
\newcommand{\res}{\mathop{\hbox{
                                \vrule height 7pt width .5pt depth 0pt
                                \vrule height .5pt width 6pt depth 0pt}}
                                \nolimits}

\newtheorem{theorem}{Theorem}[section]
\newtheorem{proposition}[theorem]{Proposition}

\newtheorem{definition}[theorem]{Definition}

\newtheorem{remark}[theorem]{Remark}

\newcommand{\defeq}{:=}

\def\la{\langle}
\def\ra{\rangle}

\title[Stationary points for average distance functional]{Stationary configurations for the average distance functional and related problems}

\author{G.~Buttazzo}
\address{Dipartimento di Matematica, Universit\`{a} di Pisa, Largo B.~Pontecorvo~5, 56127 Pisa, Italy}
\email{\tt buttazzo@dm.unipi.it}

\author{E.~Mainini}
\address{Scuola Normale Superiore, Piazza dei Cavalieri 7, I-56126 Pisa, Italy}
\email{\tt edoardo.mainini@sns.it}

\author{E.~Stepanov}
\address{Dipartimento di Matematica, Universit\`{a} di Pisa, Largo B.~Pontecorvo~5, 56127 Pisa, Italy}
\email{\tt stepanov.eugene@gmail.com}

\begin{document}
\maketitle

\begin{abstract}
For a functional 
defined on the class of closed
one-dimensional connected subsets of $\R^n$ we consider the
corresponding minimization problem and we give suitable first
order necessary conditions of optimality. The cases studied here
are the average distance functional arising in the mass
transportation theory, and the energy related to an elliptic PDE.
\end{abstract}

\vspace{0.5cm}

\textbf{AMS Subject Classification:} 49K10, 49K20, 49Q10 (primary); 49K27, 49Q20 (secondary)

\textbf{Keywords:} average distance functional, Euler equation, stationary point

\vspace{0.5cm}


\section{Introduction}\label{intro}

In this paper we consider functionals $\F(\Sigma)$ defined on the class of all closed
connected subsets of $\R^n$ and the corresponding minimization problems
\begin{equation}\label{minpb}
\min\big\{\F(\Sigma)\ :\ \Sigma\hbox{ closed connected subset of }\R^n\big\}.
\end{equation}
Due to the fact that the class of closed connected sets has good
compactness properties with respect to the Hausdorff convergence,
mild coercivity assumptions on $\F$ give the existence of
minimizers for problem~(\ref{minpb}). We are interested in finding
``first order'' necessary optimality conditions satisfied by
the minimizers $\Sigma$ of~(\ref{minpb}).

The case we consider is the \emph{average distance functional}
\begin{equation}\label{functional}
\F(\Sigma):=\int_{\R^n}\dist(x,\Sigma)\,d\mu(x)+\lambda\H(\Sigma),
\end{equation}
where $\mu$ is a given finite
nonnegative Borel measure over $\R^n$
with compact support, and the penalization term
$\lambda\H(\Sigma)$ with $\lambda>0$ is added to give a suitable
coercivity to $\F$ and to prevent minimizing sequences to spread
over all the space. A simple and standard argument involving
Blaschke and Go\l\c ab theorems gives the existence of minimizers
of $\F$. Of particular interest for us will be situations when
$\mu$ is a uniform measure over some open set $\Omega\subset\R^n$,
i.e. $\mu=\Ll^n\res \Omega$.

The average distance term in~(\ref{functional}) comes from mass
transport theory and describes for instance the total transportation
cost to move a mass $\mu$ of residents to a public transport network
$\Sigma$. This last is the unknown of the problem and has to be
designed in order to minimize $\F$, also taking into account the
construction costs which here are taken as proportional to
$\H(\Sigma)$. The minimization problem ~(\ref{minpb}), as well as
some qualitative properties of its minimizers, have been studied in
several recent papers (see
e.g.~\
Buttazzo, Oudet, Stepanov, 2002, Buttazzo, Santambrogio, 2007, Buttazzo, Stepanov, 2003, 2004, 
Paolini, Stepanov, 2004, Santambrogio, Tilli, 2004, Stepanov, 2006)
to which we refer the interested reader.  Our goal is to find
``first order'' conditions of differential character satisfied by
the minimizers of ~(\ref{functional}). Such conditions will open the
way to define a natural notion of stationary (or critical) points of
~(\ref{functional}). The main difficulty, which is quite common in
shape optimization problems, is that the domain of definition of
this functional (i.e. the class of closed connected subsets of
$\R^n$) does not possess any natural differentiable structure, and
the usual ``first variation'' argument has to be intended in a
suitable way.

In the last section we consider a similar case arising from the theory of elliptic equations:
\begin{equation}\label{ellpb1}
\F(\Sigma)\defeq  \int_\Omega u_\Sigma(x)f(x)\,dx+\lambda\H(\Sigma),
\end{equation}
where $\Omega\subset\R^2$ is a given bounded open subset, $f$ is a
given $L^2(\Omega)$ function, and $u_\Sigma$ is the unique solution
of the PDE
\[
\left\{\begin{array}{rcl}
-\Delta u&=f& \hbox{ in }\Omega\setminus\Sigma,\\
u& =0& \hbox{ on }\partial\Omega\cup\Sigma.
\end{array}\right.
\]
One has to remark that while a lot of properties are known for
minimizers of the average distance functional
(see~\ Buttazzo, Oudet, Stepanov, 2002, Buttazzo, Stepanov, 2003, Paolini, Stepanov 2004, Stepanov, 2006),
 like partial
regularity, absence of loops, topological properties (finite
number of branching points, each of which is a regular tripod), no
such property has been studied for minimizers of~(\ref{ellpb1}).

\section{The Euler equation for the average distance functional}\label{seuler}

For a compact set $\Sigma\subset \R^n$ we denote by $\pi^\Sigma$
the projection map to $\Sigma$ (i.e.\ such that
$\pi^\Sigma(x)\in\Sigma$ is one of the nearest point in $\Sigma$ to $x\in
\R^n$). This map is uniquely defined everywhere except the \emph{ridge} set
$\mathcal{R}_\Sigma$, which is defined as the set of all
$x\in\R^n$ for which the minimum distance to $\Sigma$ is attained
at more than one point. It is well known that
$\mathcal{R}_\Sigma$ is the set of non differentiability points
of the distance function to $\Sigma$ (that is, of the map
$x\in\R^n\mapsto\dist(x,\Sigma)$), and since the latter map is
semiconcave, this set is an $(\HH^{n-1},n-1)$-rectifiable Borel
set (see Proposition~3.7 in Mantegazza, Mennucci, 2003).

We will denote by $B_r(x)\subset\R^n$ the open ball with radius
$r>0$ and center $x\in \R^n$. The line segment with endpoints $A$
and $B$ will be denoted by $\overline{AB}$, the arc of a curve with
the same endpoints will be denoted by $\widetilde{AB}$ (usually in
this paper we will deal with arcs of circle).

To be begin with, we estimate the ascending local slope of~(\ref{functional})
defined by
\[
|\F'|(\Sigma)\defeq\limsup_{d_H(\Sigma',\Sigma)\to 0}
\frac{(\F(\Sigma')-\F(\Sigma))^+}{d_H(\Sigma',\Sigma)},
\]
where $d_H$ stands for Hausdorff distance between sets. The following simple assertion is valid.

\begin{proposition}\label{slope-euler}
If $\mu(\mathcal{R}_\Sigma)=0$, one has $|\F'|(\Sigma)\geq \lambda$.
\end{proposition}

\begin{proof}
Let $x\in \Sigma$ be such that $\mu((\pi^\Sigma)^{-1}(\{x\}))=0$
(all but a countable number of points of $\Sigma$ have this
property). Let then $\Sigma_\eps:=\Sigma\cup I_\eps$, where
$I_\eps$ stands for the line segment of length $\eps>0$, with one
of the endpoints $x$ and such that $\pi^\Sigma(I_\eps)=x$. Then
$d_H(\Sigma_\eps,\Sigma)=\eps$ and
$\H(\Sigma_\eps)=\H(\Sigma)+\eps$. On the other hand, denoting
\[
G_\eps\defeq \{z\in \R^n\,:\, \dist(z,\Sigma)\geq\dist(z,I_\eps)\},
\]
we have that
\[\begin{array}{lll}\vspace{6pt}
\displaystyle\int_{\R^n} \dist(x,\Sigma)\,d\mu(x)&\displaystyle\geq \int_{\R^n}
\dist(x,\Sigma_\eps)\,d\mu(x)\\&\displaystyle\geq \int_{\R^n}
\dist(x,\Sigma)\,d\mu(x) -\eps \mu(G_\eps).
\end{array}\]
Thus
\[
|\F'|(\Sigma)\geq \displaystyle
\limsup_{d_H(\Sigma_\eps,\Sigma)\to 0}
\frac{(\F(\Sigma_\eps)-\F(\Sigma))^+}{d_H(\Sigma_\eps,\Sigma)}
\geq  \displaystyle \lim_{\eps\to 0^+}\frac{(\lambda\eps-\eps
\mu(G_\eps))^+}{\eps},
\]
and to conclude the proof it suffices to mind that
$\mu(G_\eps)=o(1)$, because $G_\eps \searrow \{x\}$ as $\eps\to
0^+$.
\end{proof}

The above proposition in fact means that for the
functional~(\ref{functional}) no set $\Sigma$ (not even a minimizer)
is stationary in the strong sense, i.e.\ is such that
\[
\F(\Sigma')=\F(\Sigma)+o(d_H(\Sigma, \Sigma'))
\]
as $\Sigma'\to \Sigma$ in Hausdorff distance. Therefore, in search
for the natural notion of stationary points of $\F$ we have
to restrict the set of admissible variations of $\Sigma$. For this
purpose let  $\phi_\eps\colon \R^n\to \R^n$ be a one parameter group
of diffeomorphisms satisfying
\begin{equation}\label{generator}
    \phi_\eps(x)=x+\eps X(x)+o(\eps),
\end{equation}
as $\eps\to 0$, where $X\in C^\infty_0(\R^n;\R^n)$. We will write Euler
equation for the functional~(\ref{functional}) by considering
admissible variations of the type $\Sigma_\eps
\defeq \phi_\eps(\Sigma)$.

We recall the notion of  generalized mean curvature
(from Bouchitt\'e, Buttazzo, Fragal\`a, 1997). The generalized mean curvature
$H_\Sigma$ of a countably $(\HH^k, k)$-rectifiable set $\Sigma\subset \R^n$
(or, in terms of the above reference, of the measure
$\HH^k\res\Sigma$) is the vector-valued distribution defined by the
relationship
\[
\la X,H_\Sigma\ra := -\int_{\Sigma} \div^{\Sigma}\, X\, d\HH^k\,
\]
for all $X\in C_0^\infty(\R^n,\R^n)$, where $\div^\Sigma$ stands for
the tangential divergence operator (i.e.\ projection of the
divergence to the approximate tangent space of $\Sigma$ at
$\HH^k$-a.e.\ point of $\Sigma$).
We have then the
following result.

\begin{theorem}\label{euler}
Let $\mu$ be a Borel measure such that \[
\mu(E)=0 \mbox{ whenever }
\HH^{n-1}(E) <+\infty. 
\]
Then for
all $X\in C^\infty_0(\R^n;\R^n)$ one has
\begin{equation}\label{main0}
\begin{array}{rl}\vspace{6pt}\displaystyle
\frac{\partial}{\partial\eps}\F(\Sigma_\eps){\Big|}_{\eps=0}&\displaystyle=\int_{\R^n}
\left\la
X(\pi^\Sigma(x)),\frac{\pi^\Sigma(x)-x}{|\pi^\Sigma(x)-x|}\right\ra
\,d\mu -\lambda\la H_\Sigma,X\ra \\
&\displaystyle = \int_{\R^n} \left\la X(\pi^\Sigma(x)),\nabla \dist
(x,\Sigma)\right\ra \,d\mu -\lambda\la H_\Sigma,X\ra.
\end{array}
\end{equation}
In particular, if $\Sigma$ is a minimizer of $\F$, then
\begin{equation}\label{main}
\begin{array}{rl}\displaystyle\int_{\R^n}
\left\la
X(\pi^\Sigma(x)),\frac{\pi^\Sigma(x)-x}{|\pi^\Sigma(x)-x|}\right\ra
\,d\mu =\lambda\la H_\Sigma,X\ra
\end{array}
\end{equation}
for all $X\in C^\infty_0(\R^n;\R^n)$.
\end{theorem}

\begin{proof}
First of all, we perform the variation for the first term. We adopt
the method of calculation of the derivative of the distance function
with respect to the variation of the set, used
in Lemma 4.5 of Ambrosio, Mantegazza (1998).
Clearly, for $z:=\phi_\eps(\pi^\Sigma(x))$ one has
\[
\begin{array}{rl}\vspace{6pt}
\displaystyle\dist(x,\Sigma)&\displaystyle =|\pi^\Sigma(x)-x|,\\\displaystyle
\dist(x,\Sigma_\eps) &\displaystyle\leq |z-x|.
\end{array}
\]
 From ~(\ref{generator}) we get, for $\eps\to 0$,
\[
\begin{array}{rl}\vspace{6pt}
\displaystyle
|z-x|^2&\displaystyle= \left\la\pi^\Sigma(x)-x+\eps
X(\pi^\Sigma(x)),\pi^\Sigma(x)-x+\eps
X(\pi^\Sigma(x))\right\ra +o(\eps)\\\vspace{6pt}
&\displaystyle=|\pi^\Sigma(x)-x|^2+2\left\la\pi^\Sigma(x)-x,\eps X(\pi^\Sigma(x))\right\ra+o(\eps)\\
&\displaystyle=|\pi^\Sigma(x)-x|^2\left(1+2\left\la\frac{\pi^\Sigma(x)-x}{|\pi^\Sigma(x)-x|^2},\eps
X(\pi^\Sigma(x))\right\ra+o(\eps)\right).
\end{array}
\]
Then
\[
\begin{array}{rl}\vspace{6pt}
\displaystyle
 \dist(x,\Sigma_\eps)-\dist(x,\Sigma)&\displaystyle\leq |z-x|-|\pi^\Sigma(x)-x|\\
&\displaystyle=\eps\left\la\frac{\pi^\Sigma(x)-x}{|\pi^\Sigma(x)-x|},
X(\pi^\Sigma(x))\right\ra +o(\eps),
\end{array}
\]
and we deduce
\begin{equation}\label{first}
    \limsup_{\eps\to 0}\frac{1}{\eps}(\dist(x,\Sigma_\eps)-\dist(x,\Sigma))\leq
     \left\la\frac{\pi^\Sigma(x)-x}{|\pi^\Sigma(x)-x|}, X(\pi^\Sigma(x))\right\ra.
\end{equation}
On the other hand, consider a sequence $\eps_\nu\to 0^+$ for
$\nu\to\infty$. The set of points $x\in \R^n$ for which  both
$\pi^\Sigma(x)$ 
and $\pi^{\Sigma_{\eps_\nu}}(x)$ are singletons for any $\nu\in \N$
is of full measure $\mu$ in $\R^n$ (the complement is a countable
union of ridge sets $\mathcal{R}_{\Sigma_\nu}$ and $\mathcal{R}_{\Sigma}$ which are
all $(\HH^{n-1}, n-1)$-rectifiable, hence $\mu$-negligible). For all such $x$, since $\phi_\eps$
is invertible for all sufficiently small $\eps$, let
$\zeta:=\phi_{\eps_\nu}^{-1}(\pi^{\Sigma_{\eps_\nu}}(x))$, so that
 \[ \begin{array}{rl}\vspace{6pt}
\displaystyle\dist(x,\Sigma_{\eps_\nu})&\displaystyle=|\phi_{\eps_\nu}(\zeta)-x| ,\\\displaystyle
\dist(x,\Sigma)&\displaystyle\leq |\zeta-x|.
\end{array} \]
Again we have
\[ \begin{array}{l}\vspace{6pt}\displaystyle
 |\phi_{\eps_\nu}(\zeta)-x|-|\zeta-x|\\
  \displaystyle\qquad\qquad=|\zeta -x|\left(
 \sqrt{1+2\left\la\frac{\zeta-x}{|\zeta-x|^2},\eps_\nu
 X(\zeta)\right\ra
  +o(\eps_\nu)}-1\right)\\
  \qquad\qquad\displaystyle=\eps_\nu\left\la\frac{\zeta-x}{|\zeta-x|},
 X(\zeta)\right\ra + o(\eps_\nu).
\end{array} \]
Therefore,
\[
\dist(x,\Sigma_{\eps_\nu})-\dist(x,\Sigma)\geq
\eps_\nu\left\la\frac{\zeta-x}{|\zeta-x|},
 X(\zeta)\right\ra + o(\eps_\nu).
\]
Passing to the limit as $\nu\to \infty$, we get
\begin{equation}\label{second}
     \left\la\frac{\pi^\Sigma(x)-x}{|\pi^\Sigma(x)-x|}, X(\pi^\Sigma(x))\right\ra
     \leq\liminf_{\nu\to\infty}\frac{1}{\eps_\nu}
     \left(\dist(x,\Sigma_{\eps_\nu})-\dist(x,\Sigma)\right).
\end{equation}
Combining~(\ref{first}) with~(\ref{second}), we get for $\mu$-a.e.\
$x\in\R^n$,
\[
    \lim_{\nu\to\infty}\frac{1}{\eps_\nu}
    (\dist(x,\Sigma_{\eps_\nu})-\dist(x,\Sigma))=
    \left\la\frac{\pi^\Sigma(x)-x}{|\pi^\Sigma(x)-x|}, X(\pi^\Sigma(x))\right\ra,
\]
so that, by Lebesgue dominated convergence theorem,
\[
\begin{array}{l}
    \displaystyle\lim_{\nu\to\infty}\frac{1}{\eps_\nu}
    \int_\Om(\dist(x,\Sigma_{\eps_\nu})-\dist(x,\Sigma))\,d\mu\\
    \qquad\qquad\qquad\displaystyle=
    \int_\Omega\left\la\frac{\pi^\Sigma(x)-x}{|\pi^\Sigma(x)-x|}, X(\pi^\Sigma(x))\right\ra
    \,d\mu.
\end{array}
\]
Since the sequence $\eps_\nu$ is arbitrary, one has
\[
\begin{array}{l}
    \displaystyle    \lim_{\eps\to 0^+}\frac{1}{\eps}
    \int_\Om(\dist(x,\Sigma_{\eps_\nu})-\dist(x,\Sigma))\,d\mu\\
    \qquad\qquad\qquad\displaystyle=
    \int_\Omega\left\la\frac{\pi^\Sigma(x)-x}{|\pi^\Sigma(x)-x|}, X(\pi^\Sigma(x))\right\ra
    \,d\mu.
    \end{array}
\]
Finally, we observe that according to the Theorem 7.31
of Ambrosio, Fusco, Pallara (2000) one has
\[
\frac{d\,}{d\varepsilon}\HH^k (\Sigma_\varepsilon) = \int_{\Sigma}
\mbox{div}^{\Sigma}\, X\, d\HH^k= - \la H_\Sigma, X \ra,
\]
which concludes the proof.
\end{proof}
\begin{remark}
The assumptions of the above theorem are satisfied, in particular,
when $\mu\ll\Ll^n$.
\end{remark}

We are in a position to give the following definition.

\begin{definition}\label{def_dist_wkstat}
A closed connected set $\Sigma\subset \R^n$ will be called
stationary for the functional $\F$, if~(\ref{main}) holds.
\end{definition}

Clearly, every stationary point depends on the problem data, which
in this case is the measure $\mu$. To emphasize this dependence,
we will further sometimes say for stationary points for the
functional $\F$ that they are stationary with respect to
$\mu$. In the most important particular case we will be interested
in $\mu$ is a uniform measure over some open $\Omega\subset \R^n$
(i.e.\ $\mu=\Ll^n\res \Omega$) with $\Sigma\subset\Omega$. In such a situation we will be speaking of stationary points with respect to the set $\Omega$.

\section{Examples of regular stationary points}

We will first show that, in sharp contrast with minimizers,
stationary points may contain closed loops (i.e. homeomorphic images of
$S^1$).

\begin{proposition}
Let $\mu\defeq\Ll^2\res B_1(0)$. There exist $r<1$ such
that the circumference $\partial B_r(0)$ is a stationary point for
functional~(\ref{functional}) if and only if $\lambda <\frac12$.
Nevertheless, no circumference is a minimizer
of~(\ref{functional}), since minimizers cannot contain closed
loops.
\end{proposition}

\begin{proof}
We set $\Sigma\defeq\partial B_r(0)$ and impose~(\ref{main}). We
choose $X$ to be normal to $\Sigma$ without loss of generality,
since the normal part only plays a role in ~(\ref{main}). If we write
the integral term in polar coordinates, the integrand depends only
on the angle. Setting $A=B_r(0)$ and $B=B_1(0)\backslash B_r(0)$,
and letting $\nu(x)$ be the outward unit normal to $\partial B_r(0)$, we get
\[ \begin{array}{rl}\vspace{6pt}\displaystyle
\int_\Omega \left\la
X(\pi^\Sigma(x)),\frac{\pi^\Sigma(x)-x}{|\pi^\Sigma(x)-x|}\right\ra
\,dx &\displaystyle= \int_A\la X(\pi^\Sigma(x)),\nu(\pi^\Sigma(x))\ra \,dx
\\&\displaystyle- \int_B \la X(\pi^\Sigma(x)),\nu(\pi^\Sigma(x))\ra \,dx,
\end{array} \]
and we can compute
\[ \begin{array}{rl}\vspace{6pt}
\displaystyle\int_A\la X(\pi^\Sigma(x)),\nu(\pi^\Sigma(x))\ra \,dx &\displaystyle=\int_0^r
\int_0^{2\pi} |X(\theta)|\,\rho\: d\rho d\theta\\&\displaystyle =\frac12 r^2
\int_0^{2\pi} |X(\theta)|d\theta,
\end{array} \]
and similarly for the integral over $B$. Moreover,
\[
    \la X,H_\Sigma\ra=-\int_{\partial B_r(0)}|H_\Sigma(x)|\la X(x),\nu(x)\ra d
    \H(x)=-\frac 1 r \int_0^{2\pi}|X(\theta)|r d\theta.
\]
So the Euler equation reads
\begin{equation}\label{Eulercirc}
    \left(
    r^2-\frac12+{\lambda}\right)\int_0^{2\pi}|X(\theta)|d\theta=0.
\end{equation}
This equation is identically satisfied, if and only if $\lambda
<1/2$, for $r=\sqrt{1/2 -\lambda}$ (of course, $\lambda=1/2$ would
also suite for~(\ref{Eulercirc}), but it corresponds to a
degenerate case when the circumference reduces to a point).

To show that minimizers of~(\ref{functional}) cannot contain closed
loops, and hence the above stationary points are not minimizers, we
may act as in the proof of absence of loops in minimizers of average
distance functionals with length constraint (see
e.g. Paolini, Stepanov, 2004,  Buttazzo, Oudet, Stepanov 2002 or Buttazzo, Stepanov, 2003). In fact,
suppose that $\Sigma$ is a minimizer containing a closed loop. Then
there is a set of positive length $C\subset \Sigma$ such that for
every $x\in C$ and for every $\eps>0$ there is a closed connected
subset $D_\eps\subset \Sigma$ such that $x\in D_\eps$, $\diam D_\eps
=\eps$ (hence $\H(D_\eps)\geq \eps$) and $\Sigma_\eps\defeq
\Sigma\setminus D_\eps$ is connected. We may suppose without loss of
generality that $\mu((\pi^\Sigma)^{-1}(\{x\}))=0$ for all $x\in C$
(since the set of atoms of the latter measure is clearly at most
countable). One has then by triangle inequality
\[
\int_{\R^n} \dist(x,\Sigma_\eps)\,d\mu(x)\leq \int_{\R^n}
\dist(x,\Sigma)\,d\mu(x) +\eps\mu((\pi^\Sigma)^{-1}(D_\eps)),
\]
and hence
\[
\F(\Sigma_\eps)\le\F(\Sigma)+\eps
\mu((\pi^\Sigma)^{-1}(D_\eps)) -\lambda \eps.
\]
Minding that $D_\eps\searrow \{x\}$ as $\eps\to 0^+$, we get
\[
\mu((\pi^\Sigma)^{-1}(D_\eps)) \to \mu((\pi^\Sigma)^{-1}(\{x\})) =0,
\]
and thus
\[
\F(\Sigma_\eps)\le\F(\Sigma)+o(\eps)-\lambda\eps
\]
as $\eps\to 0^+$, which means that $\F(\Sigma_\eps) < \F(\Sigma)$
for small $\eps >0$ concluding the proof.
\end{proof}

Let us now consider another  example of a stationary point
for~(\ref{functional}) given by Figure~\ref{Figura 1}, where the
radii of the semicircles are equal to $\sqrt{\lambda}$. Here, as
well as in all the other figures, the arrows starting at the
endpoints of $\Sigma$ indicate the directions of $-H_\Sigma$ in
these points.

\begin{proposition}\label{prop_stazsegment}
There exists a line segment which is stationary for the region $\Omega$
shown on Figure~\ref{Figura 1}.
\end{proposition}

\begin{proof}
In the example of Figure~\ref{Figura 1}, points belonging to regions $A$
and $B$ are projected to the line segment $\Sigma$ along the perpendicular, and it is clear that the symmetry of the domain yields
\[
\int_A \left\la
X(\pi^\Sigma(x)),\frac{\pi^\Sigma(x)-x}{|\pi^\Sigma(x)-x|}\right\ra
\,dx + \int_B \left\la
X(\pi^\Sigma(x)),\frac{\pi^\Sigma(x)-x}{|\pi^\Sigma(x)-x|}\right\ra
\,dx=0
\]
for any vector field $X\in C^\infty_0(\R^n;\R^n)$.

Set $X_1:=\la X,\mathbf{e}_1\ra$ and $X_2:=\la X,\mathbf{e}_2\ra$,
where $\mathbf{e}_1$, $\mathbf{e}_2$ stand for the base vectors in
$\R^2$. Let us compute the contribution of the right unit
semicircle:
\[ \begin{array}{rl}\vspace{6pt}\displaystyle
\int_D \left\la
X(\pi^\Sigma(x)),\frac{\pi^\Sigma(x)-x}{|\pi^\Sigma(x)-x|}\right\ra
\,dx
&\displaystyle=-\int_0^{\sqrt{\lambda}}\int_{-\pi/2}^{\pi/2}X_1(F)\cos\theta
\rho
d\rho d\theta\\\vspace{6pt}
&\displaystyle=-2X_1(F)\int_0^{\sqrt{\lambda}}\rho d\rho\\&=-\lambda X_1(F).
\end{array} \]
In the same way, the contribution of semicircle $C$ is given by
\[ \begin{array}{rl}\displaystyle
\int_C \left\la
X(\pi^\Sigma(x)),\frac{\pi^\Sigma(x)-x}{|\pi^\Sigma(x)-x|}\right\ra
\,dx =\lambda X_1(E).
\end{array} \]
Therefore,
\[
\int_\Omega \left\la
X(\pi^\Sigma(x)),\frac{\pi^\Sigma(x)-x}{|\pi^\Sigma(x)-x|}\right\ra
\,dx =-\lambda X_1(F) + \lambda X_1(E).
\]
On the other hand, at the endpoints $E$ and $F$ of the segment,
the distributional curvature is given by $\delta_E \mathbf{e}_1$,
$ -\delta_F \mathbf{e}_1$ where $\delta_x$ stands for the Dirac
mass concentrated at the point $x$ (see Bouchitt\'e, Buttazzo, Fragal\`a, 1997),
while at all the other points of the segment the curvature is
zero. Thus the curvature term of the Euler equation reduces to
$\lambda (X_1(F)-X_1(E))$, and hence~(\ref{main}) is satisfied.
\end{proof}
\begin{figure}\label{Figura 1}
\begin{center}
\includegraphics[width=12cm]{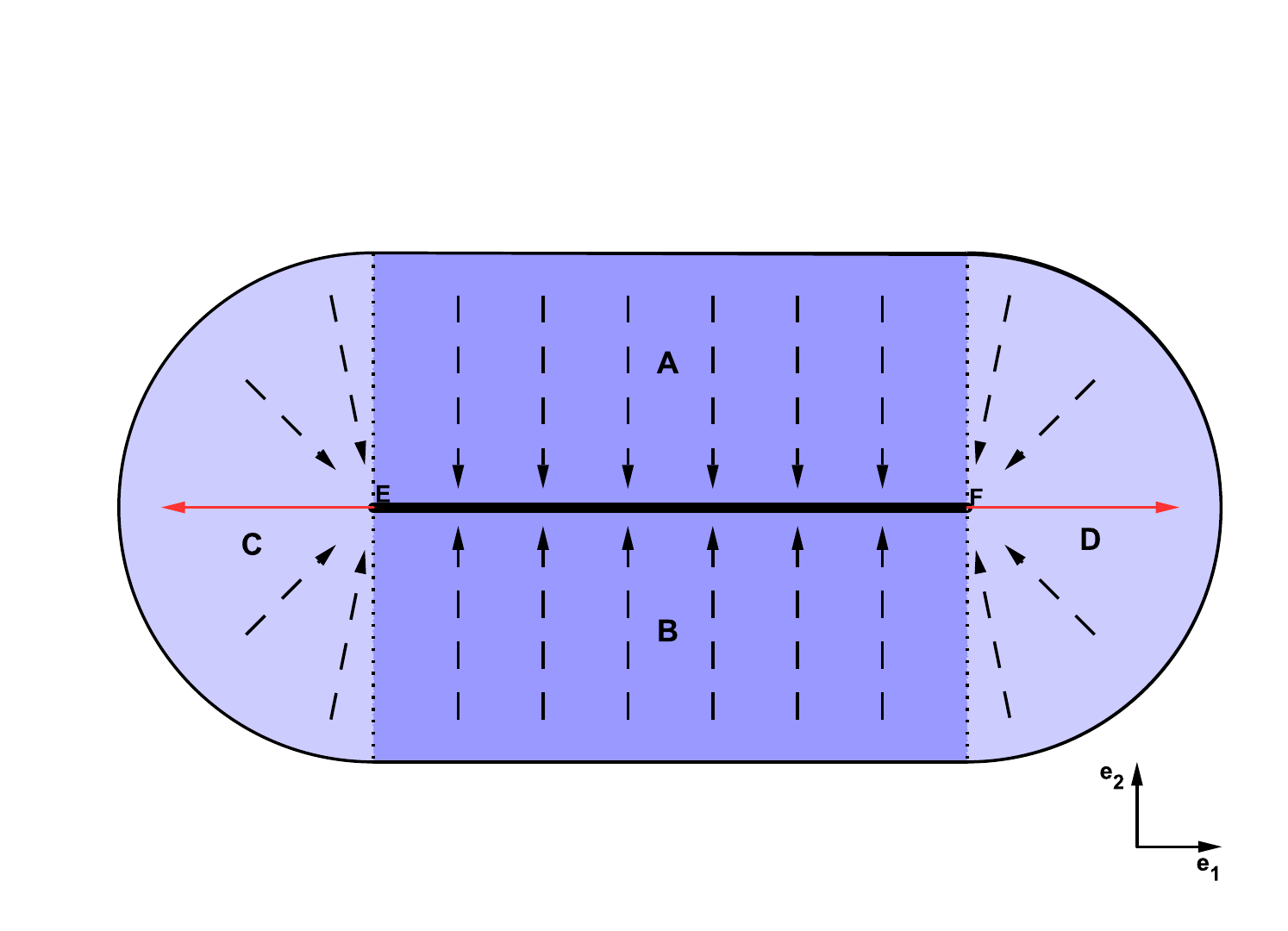}\caption{Construction of the proof of
Proposition~\ref{prop_stazsegment}}
\end{center}
\end{figure}

We now show an example of a set which is never stationary (i.e. it
is not stationary for any ambient set $\Om$).

\begin{proposition}\label{prop_nonstazangle}
The line $\Sigma$ made of two segments (not reduced to a single
segment), is not stationary for any open set $\Omega\subset \R^2$.
\end{proposition}

\begin{proof}
Let $P$ be the common vertex of the two segments (with the
aperture $2\varphi<\pi$),
 $R$ be a point on one of the two
edges, with $z\defeq |P-R|$. Let moreover $S$ be a point on the
normal to the same segment passing through $R$, with $y\defeq
|S-R|$ located in the region $B$ in Figure~\ref{Figura 2}. Since
the whole polygonal line $\Sigma$, and hence $P$, is in the
interior of
$\Om$, it is clear that 
the rectangle $B\defeq PRST$ (with sidelengths $z$ and $y$), is
all contained in $\Om$ for all sufficiently small $y$ and $z$. Let
finally $Q$ be a point of the intersection of the line passing
through $S$ and $R$, with the bisector of the angle formed by the
two segments of $\Sigma$ (see Figure~\ref{Figura 2}).
 Choose now a regular vector field $X$ compactly
supported in the open segment $\overline{PR}$, and normal to it,
pointing towards the region $B$ in Figure~\ref{Figura 2}. It is
clear that there is no contribution from the curvature term in the
Euler equation, since the curvature of the line segment is zero
outside its endpoints. So it remains to check the integral term.
Since $|Q-R|=z\tan\varphi$, an easy computation in the suitable
coordinate system yields
\[ \begin{array}{rl}\displaystyle
\int_B \left\la
X(\pi^\Sigma(x)),\frac{\pi^\Sigma(x)-x}{|\pi^\Sigma(x)-x|}\right\ra
\,dx &\displaystyle=-\int_B |X(\pi^\Sigma(x))|\,dx\\
&\displaystyle =-y\int_0^z |X(\zeta)|d\zeta
\end{array} \]
and
\[ \begin{array}{rl}\displaystyle
\int_A\left\la
X(\pi^\Sigma(x)),\frac{\pi^\Sigma(x)-x}{|\pi^\Sigma(x)-x|}\right\ra
\,dx &\displaystyle=\int_A
|X(\pi^\Sigma(x))|\,dx\\
&\displaystyle=z\tan\varphi\int_0^z|X(\zeta)|d\zeta.
\end{array} \]
Notice that $z$ can be chosen small enough such that the sum of
the above terms is strictly negative, while
\[ \begin{array}{rl}\vspace{6pt}\displaystyle
\int_\Om\left\la
X(\pi^\Sigma(x)),\frac{\pi^\Sigma(x)-x}{|\pi^\Sigma(x)-x|}\right\ra
\,dx &\displaystyle \leq \int_A \left\la
X(\pi^\Sigma(x)),\frac{\pi^\Sigma(x)-x}{|\pi^\Sigma(x)-x|}\right\ra
\,dx \\\vspace{6pt} &\displaystyle + \int_B \left\la
X(\pi^\Sigma(x)),\frac{\pi^\Sigma(x)-x}{|\pi^\Sigma(x)-x|}\right\ra
\,dx,
\end{array} \]
 that is, for sufficiently small $z$ the equation~(\ref{main})
is not satisfied.
\end{proof}
\begin{figure}
\includegraphics[width=12cm]{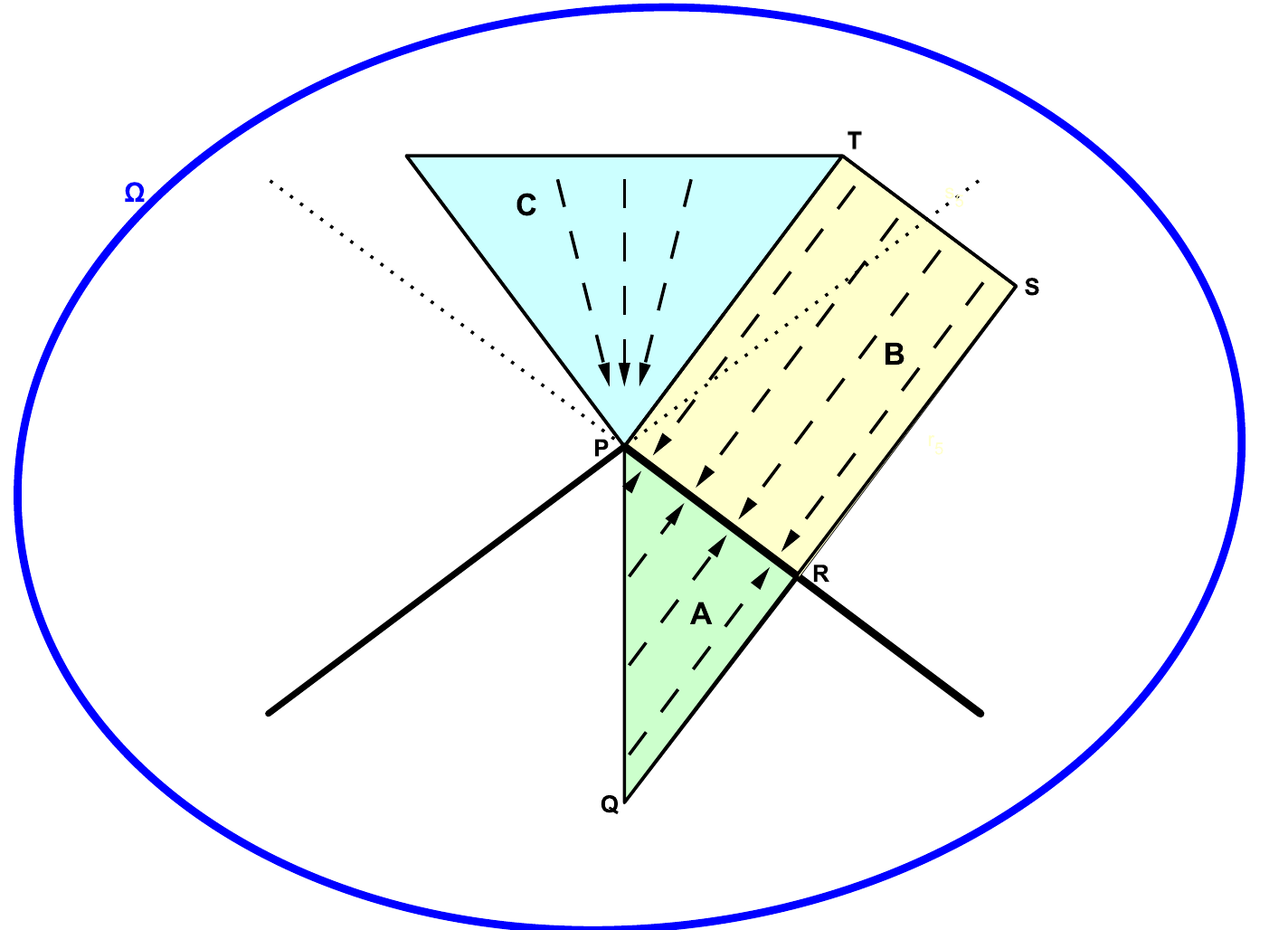}\caption{Construction of the proof of
Proposition~\ref{prop_nonstazangle}}\label{Figura 2}
\end{figure}

It is worth emphasizing that it is still quite easy to find a
measure $\mu$ such that the given polygonal line is stationary
with respect to $\mu$.

\section{Examples of irregular stationary points}

In this section we will show that there exist $\Om$ and $\Sigma$
stationary in $\Om$ such that $\Sigma$ has angular points.

From now on, we will consider sets $\Sigma$ made of two arcs of
circumference with a common end point $O$. We will refer to such
sets simply as curved corners. We will say that a curved corner is
convex, if it is a convex curve (i.e.\ it intersects every line in
at most two points). 


\begin{proposition}\label{prop_stazcorner}
There exists a convex curved corner $\Sigma$ stationary with
respect to some open $\Omega\subset \R^2$.
\end{proposition}

\begin{proof}
Let $\lambda>0$ be fixed. Our construction is that shown on
Figure~\ref{Figura 3}. Namely, the set $\Sigma$ is made by two
arcs $\widetilde{QO}$ and $\widetilde{PO}$ of circumferences with
the same radius $R$ and with centers $C_1$ and $C_2$ respectively.
The points $P$ and $Q$ are chosen in such a way that both belong
to the line $v$ containing the centers of the circumferences. We
denote by $2\varphi \in[0,\pi]$ the angle between the normals in
$O$ to the respective arcs, pointing away from $v$. Then
$\alpha=\pi/2 -\varphi$ is the angle between $v$ and the ray $C_1
O$ (and also, by symmetry, between $v$ and the ray $C_2 O$). We
also assume the unit coordinate vectors $\mathbf{e}_1$ and
$\mathbf{e}_2$ to be directed as in Figure~\ref{Figura 3}.

 Now let
\[ \begin{array}{rl}\vspace{6pt}\displaystyle
    b & \displaystyle\defeq \sqrt{R^2+2\lambda}-R,\\\vspace{6pt}\displaystyle
    f(\theta) &\displaystyle\defeq
\sqrt{2R^2+2\lambda-\left(\frac{R\cos\alpha}{\cos\theta}\right)^2}
    ,\quad
    \theta\in[0,\alpha],\\\displaystyle
    r &\displaystyle\defeq \sqrt{2\lambda}.
\end{array} \]
Notice that $r>b$. Moreover, fix a $k \in (0,R(1-\cos\alpha))$ and
an $h>0$ such that
\begin{equation}\label{h}
    -\int_{-k}^k\left(\int_{-h}^0
    y(z^2+y^2)^{-1/2}dy\right)dz ={\lambda}.
\end{equation}
 Consider now the region bounded by $\Sigma$ and the
segment $\overline{PQ}$. It is divided symmetrically in two
regions $A$ and $B$ by the line $u$ passing through $O$
perpendicular to $v$. Let $C$ indicate the region identified by
the arc $\widetilde{QO}$, the ray $C_1O$, the ray $C_1 Q$ and the
curve defined by the equation $\rho=f(\theta)$ in polar
coordinates with center $C_1$ and the angle $\theta$ counted
counterclockwise increasing from $0$ to $\alpha$. Define $D$ to be
the region  symmetric to $C$ with respect to $u$. Let $E$ and $G$
be equal rectangles with an edge on $v$ of length $k$, centered in
$P$ and $Q$ respectively, with another edge of length $h$, and
belonging to the half space bounded by $v$ and not containing $O$.
Finally, let $F$ stand for the circular sector with center $O$
 and with the radius $r$ bounded by the normals to $\widetilde{QO}$ and
$\widetilde{PO}$ as in Figure~\ref{Figura 3}.

Define now $\Om\defeq A\cup B\cup C \cup D\cup E \cup F \cup G$. We
will show that $\Sigma$ is optimal with respect to such $\Om$. Let
$\nu$ be the outward normal to $\widetilde{QO}$. Points in $B$ and
$C$ are projected on $\Sigma$ to the arc $\widetilde{QO}$, and since
$f(\alpha)=R+b$ and $f(\theta)>R+b$ for $\theta\in[0,\alpha)$, we
have
\[ \begin{array}{rl}\vspace{6pt}\displaystyle
\int_C \left\la
X(\pi^\Sigma(x)),\frac{\pi^\Sigma(x)-x}{|\pi^\Sigma(x)-x|}\right\ra
\,dx =&\displaystyle-\int_0^{\alpha} \int_R^{f(\theta)}\la
X(\theta),\nu(\theta)\ra \rho d\rho d\theta\\\vspace{6pt}\displaystyle=&\displaystyle-\int_0^{\alpha}
\int_{R+b}^{f(\theta)}\la X(\theta),\nu(\theta)\ra \rho d\rho
d\theta\\&\displaystyle-\int_0^{\alpha} \int_{R}^{R+b}\la
X(\theta),\nu(\theta)\ra \rho d\rho d\theta,
\end{array} \]
but, by the definition of $b$ and $f$,
\[ \begin{array}{l}\vspace{6pt}\displaystyle
\int_0^{\alpha} \int_{R}^{R+b}\la X(\theta),\nu(\theta)\ra \rho
d\rho d\theta \displaystyle=\left(\frac12 (R+b^2) -\frac12
R^2\right)\int_0^{\alpha}\la X(\theta),\nu(\theta)\ra d\theta\\
\qquad\qquad\qquad\qquad\qquad\displaystyle=
\lambda\int_0^{\alpha}\la X(\theta),\nu(\theta)\ra d\theta,
\end{array} \]
\[ \begin{array}{rl}\displaystyle
\int_{R+b}^{f(\theta)}\rho d\rho &\displaystyle=\frac1 2(f(\theta))^2-\frac12
(R+b)^2=\frac12
R^2-\frac12\left(\frac{R\cos\alpha}{\cos\theta}\right)^2.
\end{array} \]
For the computation of the integral in the region $B$, it is easily
seen that
\begin{equation}\label{A}
    B=\left\{(\rho,\theta)\colon 0\leq\theta\leq\alpha\;,\;\frac{R\cos\alpha}{\cos\theta}\leq\rho\leq
    R\right\},
    \end{equation}
so it follows that
\[
\begin{array}{rl}\vspace{6pt}\displaystyle
\int_B\left\la
X(\pi^\Sigma(x)),\frac{\pi^\Sigma(x)-x}{|\pi^\Sigma(x)-x|}\right\ra
\,dx \displaystyle=\int_0^\alpha\la X(\theta),\nu(\theta)\ra
\int_{\frac{R\cos\alpha}{\cos\theta}}^R\rho d\rho
d\theta\\
\qquad\qquad\qquad\qquad\qquad\displaystyle=\frac12\int_0^\alpha\la
X(\theta),\nu(\theta)\ra\left(R^2-\left(\frac{R\cos\alpha}{\cos\theta}\right)^2\right)d\theta.
\end{array}
\]
Hence one obtains
\begin{equation}\label{1}
    \int_{ B\cup C} \left\la X(\pi^\Sigma(x)),\frac{\pi^\Sigma(x)-x}{|\pi^\Sigma(x)-x|}\right\ra
    \,dx=
    - \lambda 
    \int_0^{\alpha}\la
X(\theta),\nu(\theta)\ra d\theta.
\end{equation}

 Now consider the curvature term of the Euler equation. Let
$H_\Sigma(\widetilde{QO})$ indicate the nonatomic part of the
curvature of the arc $\widetilde{QO}$, i.e. the part not involving
the contribution of
 endpoints.
 The term $\la H_\Sigma(\widetilde{QO}), X \ra$ is clearly equal
 to
\[
-\int_0^\alpha\la X(\theta),\nu(\theta)\ra d\theta.
\]
We end up with
\begin{equation}\label{thefirst}
  \int_{ B\cup C} \left\la
X(\pi^\Sigma(x)),\frac{\pi^\Sigma(x)-x}{|\pi^\Sigma(x)-x|}\right\ra
    \,dx -\lambda \la H_\Sigma(\widetilde{QO}), X \ra=0
\end{equation}

By symmetry, the integral over region $A\cup D$ can be computed in
polar coordinates with respect to $C_2$ and $v$, with angle
$\theta^\prime$ counted clockwise increasing from $0$ to $\alpha$,
and has exactly the same form. Reasoning in the same way, one sees
the analogy between the terms $\la H_\Sigma(\widetilde{PO}), X
\ra$ and $\la H_\Sigma(\widetilde{QO}), X \ra$. It follows that
\begin{equation}\label{thesecond}
    \int_{ A\cup D} \left\la
X(\pi^\Sigma(x)),\frac{\pi^\Sigma(x)-x}{|\pi^\Sigma(x)-x|}\right\ra
    \,dx -\lambda \la H_\Sigma(\widetilde{PO}), X \ra=0.
\end{equation}

 Let us now compute the integrals over $E$ and
$G$. These two regions are disjoint thanks to the choice of $k$.
By~(\ref{h}) we get
\begin{equation}\label{3}
\begin{array}{rl}\vspace{6pt}
\displaystyle \int_{E} \Big\la
X(\pi^\Sigma(x)),&\displaystyle\frac{\pi^\Sigma(x)-x}{|\pi^\Sigma(x)-x|}\Big\ra
    \,dx\\\vspace{6pt}
    &\displaystyle=\displaystyle -X_2(P)\int_{-k}^k\left(\int_{-h}^0
    y(z^2+y^2)^{-1/2}dy\right)dz
    \\&\displaystyle = \lambda
    X_2(P).
\end{array}
\end{equation}
Analogously the integral over $G$ is given by
\begin{equation}\label{thethird}
   \displaystyle \int_{G} \left\la
X(\pi^\Sigma(x)),\frac{\pi^\Sigma(x)-x}{|\pi^\Sigma(x)-x|}\right\ra
    \,dx =  \lambda 
X_2(Q).
\end{equation}

 For the integral over $F$, we consider
polar coordinates referred to the center $O$ with the angle
$\theta$ measured counterclockwise starting from the direction
parallel to the ray $C_1 Q$, so that
$$\frac{\pi^\Sigma(x)-x}{|\pi^\Sigma(x)-x|}=-(\cos\theta,\sin\theta),\quad x\in F.$$
Then, since in $F$ the minimum distance from $\Sigma$ is always
attained in the point $O$, we get
\begin{equation}\label{2}
\begin{array}{rl}\vspace{6pt}\displaystyle
\int_{F} \Big\la X(\pi^\Sigma(x)),&\displaystyle
\frac{\pi^\Sigma(x)-x}{|\pi^\Sigma(x)-x|} \Big\ra
    \,dx\\\vspace{6pt}
    &\displaystyle= -\int_{\frac\pi 2 -\varphi}^{\frac\pi 2
    +\varphi}\int_0^r\la X(O),(\cos\theta,\sin\theta)\ra \rho d\rho d\theta\\\vspace{6pt}
    &\displaystyle=
    -X_2(O)\int_{\frac\pi 2 -\varphi}^{\frac\pi 2
    +\varphi}\int_0^r \sin\theta \rho d\rho d\theta\\
    &\displaystyle=-X_2(O) r^2
    \sin\varphi=- 2\lambda 
    X_2(O)
    \sin\varphi
\end{array}
\end{equation}

Finally consider the curvature terms at the endpoints $P$ and $Q$.
We have respectively
\begin{equation} \label{thefourth}
   \la H_\Sigma(P),X\ra   =X_2(P),\quad
   \la H_\Sigma(Q),X\ra =X_2(Q).
\end{equation}
For the point $O$, we have
\[
     H_\Sigma(O)= -2\cos\alpha\,\delta_O\mathbf{e}_2,
\]
yielding
\begin{equation}\label{thefifth}
    \la H_\Sigma(O),X\ra=-2\sin\varphi X_2(O).
\end{equation}

Since $\Omega=A\cup B\cup C\cup D\cup E\cup F\cup G$ and
\[\begin{array}{rl}\vspace{6pt}\displaystyle
 \la H_{\Sigma},X\ra&\displaystyle=  \la H_{\Sigma}(\widetilde{QO}),X\ra+ \la H_{\Sigma}(\widetilde{PO}),X\ra+ \la
 H_{\Sigma}(P),X\ra\\&\quad\displaystyle+ \la H_{\Sigma}(O),X\ra+ \la
 H_{\Sigma}(Q),X\ra,
 \end{array}
\]
combining~(\ref{thefirst}), ~(\ref{thesecond}), ~(\ref{3}),
~(\ref{thethird}), ~(\ref{2}), ~(\ref{thefourth})
and~(\ref{thefifth}) we see that the Euler equation~(\ref{main})
is identically satisfied.
\end{proof}

\begin{figure}
\includegraphics[width=12cm]{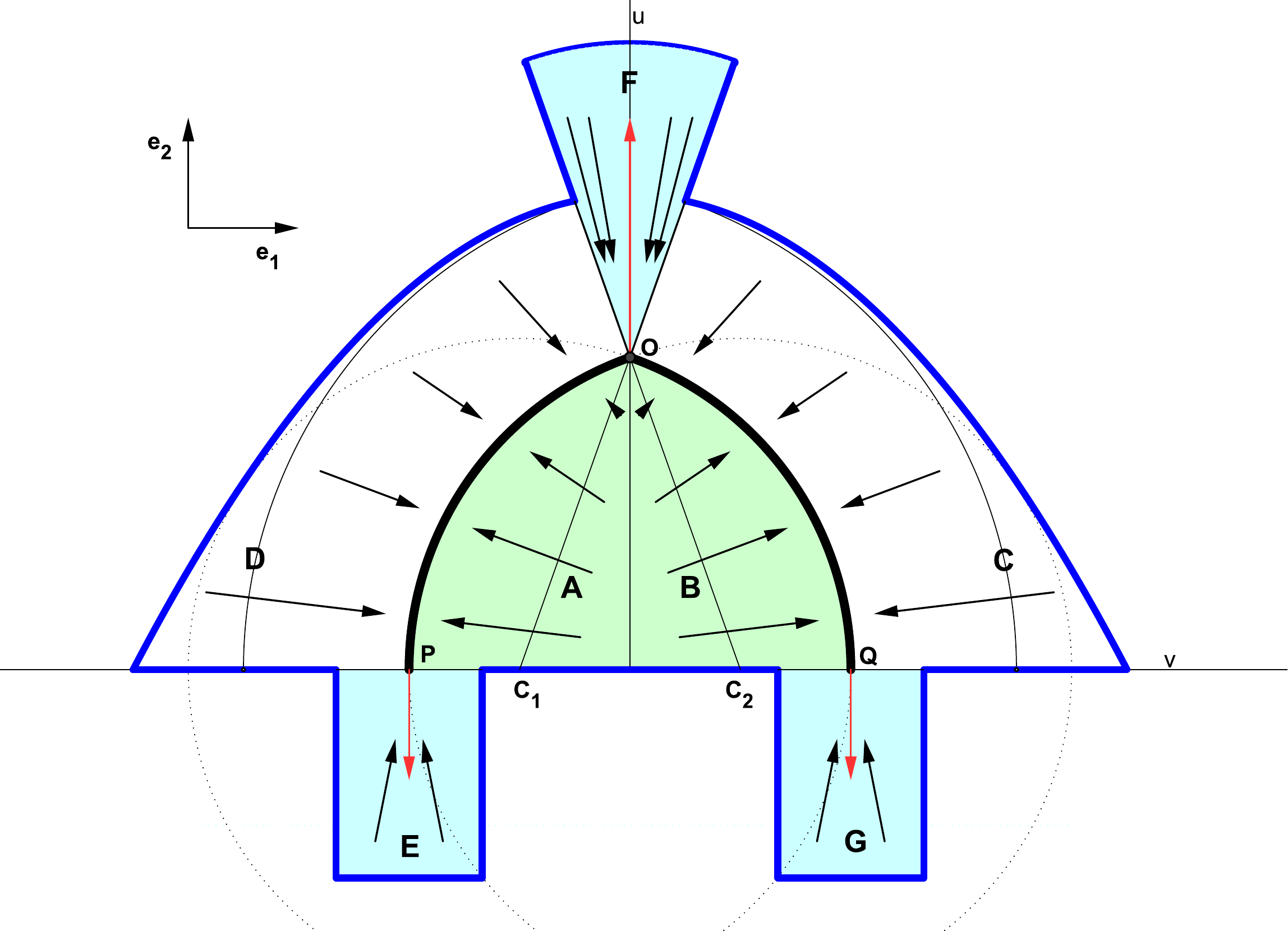}\caption{Construction of the proof of Proposition~\ref{prop_stazcorner}}
\label{Figura 3}
\end{figure}

Next we will show that, for a convex domain $\Om$, if the
amplitude of the corner is not too large, then a set composed of
two arcs of circle is not stationary.

We first introduce the notation similar to that used in the proof
of Proposition~\ref{prop_stazcorner}, but for a generic curved
corner $\Sigma$ made by two arcs $\widetilde{QO}$ and
$\widetilde{PO}$ of circumferences with different radii $R_1$ and
$R_2$ and with centers $C_1$ and $C_2$ respectively. Again
$2\varphi \in[0,\pi]$ is the angle between the normals in $O$
which bound the set of points (we will call the bisector ray of
the latter angle $u$) in $\R^2$ having $O$ as the unique point of
minimum distance to $\Sigma$. Let $v'$ be a ray starting at $C_1$
forming the angle $\alpha\leq \pi/2-\varphi$ with the ray $C_1 O$.
We assume that $\alpha$ is sufficiently small so that $v'$ meets
$\widetilde{QO}$ in some point $M$. In this way the rays $v'$,
$C_1 O$ and the arc $\widetilde{QO}$ form a sector of area $\alpha
R_1^2/2$. We also assume the unit coordinate vectors
$\mathbf{e}_1$ and $\mathbf{e}_2$ to be directed as in
Figures~\ref{Figura 4} and~\ref{Figura 5}.

Fix $\alpha$ small enough such that $v'$ meets the continuation of
$u$. Consider the ridge set $\mathcal{R}_\Sigma$ of $\Sigma$
 (i.e.\ the set of points of equal
distance from the two arcs). Note that there exists a segment
$\overline{OZ}$ with $Z\in v'$ which intersects
$\mathcal{R}_\Sigma$
only in $O$ (this 
assertion is implied by the fact that $\mathcal{R}_\Sigma$
 is a regular
curve tangent in $O$ to $u$). Denote by $W$ the curvilinear
triangle bounded by $\widetilde{MO}$, $\overline{ZO}$ and
$\overline{ZM}$. Clearly it contains the set of points $T$ having
the projection to $\Sigma$ on $\widetilde{MO}$. Moreover, they are
all projected to $\widetilde{MO}$ from the same side (i.e.\ either
from outside of the circle $B_{R_1}(C_1)$ as in Figure~\ref{Figura
4}, or from the inner part of the circle $B_{R_1}(C_1)$ as in
Figure~\ref{Figura 5}). It is important to observe that there are
no points with such a property outside of $W$. We denote by $C$
the set of points having the projection to $\Sigma$ on
$\widetilde{MO}$ but from the different side with respect to $T$.

In this section we will consider a vector field $X$ supported in a
small neighborhood of a subset of $\widetilde{MO}$ (in polar
coordinates with respect to $C_1$ and $v'$, the points of the
support are contained in the set with angular coordinate
$\theta\in[\theta_0,\alpha]$). We assume that $X$ be vanishing in
$O$ and have restriction to $\widetilde{MO}$ directed towards the
outward normal $\nu$ to the circle $B_{R_1}(C_1)$. Thus in the
first member of~(\ref{main}) the only nonzero terms are the
integrals in the regions $T$ and $C$ and the curvature term
restricted to $\widetilde{MO}$.

\begin{proposition}\label{prop_lemma}
A non convex curved corner is not stationary, for any $\Om\subset
\R^2$.
\end{proposition}

\begin{proof}
Let $\Sigma$ be a non convex curved corner. In this case one of
the centers belongs to one of the rays bounding the cone 
of points for which the projection to $\Sigma$ coincides with $O$
(let it be $C_1$). So the region $C$ is inside the sector bounded
by the arc $\widetilde{QO}$ (see Figure~\ref{Figura 4}).
 If $\beta$ is the angle formed by $\overline{OZ}$ and
$\overline{OC_1}$, it is easily seen that one can choose the point
$Z$ so that $\beta\in (\pi/2,\pi)$. In polar coordinates with
respect to $C_1$ and $v'$ for small $\alpha$ one has then
\begin{equation}\label{polarinside2}
W=\left\{(\rho,\theta)\colon 0<\theta<\alpha ,
R_1<\rho<\frac{R_1\sin\beta}{\sin(\beta+\alpha-\theta)}\right\}
\end{equation}
(observe that $\sin\beta/\sin(\beta+\alpha-\theta) >1$ since
$\beta\in (\pi/2, \pi)$, and $\alpha-\theta>0$ is small enough).
 We obtain also
 \[
   - \lambda\la H_\Sigma,X\ra={\lambda} \int_{\theta_0}^\alpha
     \la X(\theta),\nu(\theta)\ra d\theta={\lambda} \int_{\theta_0}^\alpha
    |X(\theta)|d\theta.
\]
Moreover, thanks to~(\ref{polarinside2}), we have
\[
\begin{array}{l}\vspace{6pt}\displaystyle
\left|\int_T \left\la
X(\pi^\Sigma(x)),\frac{\pi^\Sigma(x)-x}{|\pi^\Sigma(x)-x|}\right\ra
    \,dx \right| \displaystyle\leq  \int_W \left|\left\la
X(\pi^\Sigma(x)),\frac{\pi^\Sigma(x)-x}{|\pi^\Sigma(x)-x|}\right\ra\right|
    \,dx \\
    \qquad\qquad\qquad\qquad\qquad\displaystyle=\frac12\int_{\theta_0}^\alpha {|X(\theta)|} \left(\frac{R_1^2\sin^2\beta}{\sin^2(\beta+\alpha-\theta)}-R_1^2\right)d
    \theta.
\end{array}
\]
But for $\theta\to \alpha$,  with $\beta$ fixed, we get
\[
\frac{R_1^2\sin^2\beta}{\sin^2(\beta+\alpha-\theta)}-R_1^2=o(1),
\]
implying that
\[
\left|\int_T \left\la
X(\pi^\Sigma(x)),\frac{\pi^\Sigma(x)-x}{|\pi^\Sigma(x)-x|}\right\ra
    \,dx\right|\leq o\left(\int_{\theta_0}^\alpha |X(\theta)|\, d\theta\right).
\]
Therefore it is clear that, for $\theta_0$ close enough to $\alpha$,
the Euler equation~(\ref{main}) is never satisfied for $\Sigma$. In
fact, since
\[
\int_C \left\la
X(\pi^\Sigma(x)),\frac{\pi^\Sigma(x)-x}{|\pi^\Sigma(x)-x|}\right\ra
\geq 0,
\]
 we have that
\[ \begin{array}{l}\vspace{6pt}\displaystyle
\int_\Om \left\la
X(\pi^\Sigma(x)),\frac{\pi^\Sigma(x)-x}{|\pi^\Sigma(x)-x|}\right\ra
-\lambda \la H_\Sigma, X\ra \\
\qquad\qquad\qquad\qquad\qquad \geq \displaystyle \int_T \left\la
X(\pi^\Sigma(x)),\frac{\pi^\Sigma(x)-x}{|\pi^\Sigma(x)-x|}\right\ra
-\lambda \la H_\Sigma, X\ra.
\end{array} \]
Hence, for $\theta_0\to \alpha$ one has
\[ \begin{array}{rl}\vspace{6pt}\displaystyle
\int_\Om \left\la
X(\pi^\Sigma(x)),\frac{\pi^\Sigma(x)-x}{|\pi^\Sigma(x)-x|}\right\ra
-\lambda \la H_\Sigma, X\ra \geq &\displaystyle \lambda \int_{\theta_0}^\alpha
|X(\theta)|\, d\theta \\
& \displaystyle- o\left(\int_{\theta_0}^\alpha |X(\theta)|\, d\theta\right),
\end{array} \]
that is, the right hand side of the above inequality is always
strictly positive once $\theta_0$ is sufficiently close to
$\alpha$.
\end{proof}
\begin{figure}
\includegraphics[width=12cm]{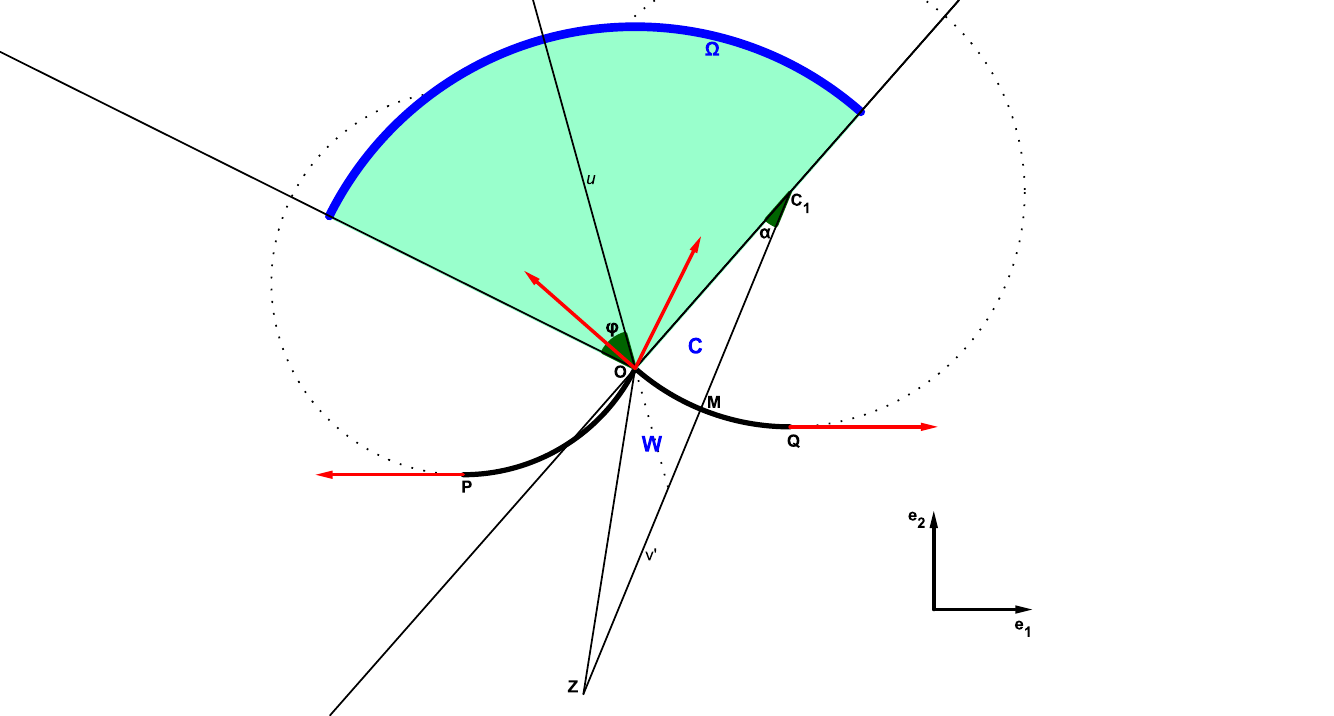}\caption{Construction of the proof of Proposition~\ref{prop_lemma}}\label{Figura 4}
\end{figure}

Finally, we show that the condition for a curved corner to be
stationary with respect to a \emph{convex} $\Om$ is even more
restrictive.

\begin{proposition}\label{prop_43} Let $\Om$ be convex.
 A curved corner is
not stationary with respect to $\Om$ if
\begin{equation}\label{inequality}
\frac{4\lambda}{
    b_1^2+b_2^2}\geq h(\varphi),
\end{equation}
where $b_i\defeq \sqrt{R_i^2+2\lambda}-R_i$, $i=1,2$,
\[
h(\varphi)\defeq \frac{1}{\sin \varphi}\int_0^\varphi
\frac{\cos(\varphi-\theta)}{\cos^2 \theta}d\theta.
\]  In particular, there are no curved corners of amplitude less
than or equal to $2\gamma$, where $\gamma\in (0,\pi/2)$ is the
angle that satisfies
$$
\int_0^\gamma \frac{\cos(\gamma-\theta)}{\cos^2
\theta}d\theta=\sin\gamma,
$$
so $\gamma\simeq 54^\circ$.
\end{proposition}

\begin{proof} If the curved corner is not convex, we refer to the
previous Proposition~\ref{prop_lemma}. Otherwise, let $\beta$ be
the angle between $\overline{OZ}$ and $\overline{OC_1}$. This time
$\beta<\varphi$, so that once $\alpha$ is sufficiently small, one
has $\beta+\alpha<\pi/2$ (see Figure~\ref{Figura 5}). In polar
coordinates with respect to $C_1$ and $v'$, we have
\begin{equation}\label{polarinside}
W=\left\{(\rho,\theta)\colon 0<\theta<\alpha ,
\frac{R_1\sin\beta}{\sin(\beta+\alpha-\theta)}<\rho<R_1\right\}.
\end{equation}
Notice that $\alpha-\theta>0$ is small  and the bound on $\beta$
gives $\sin\beta/\sin(\beta+\alpha-\theta) <1$. The curvature term
in the Euler equation~(\ref{main}) is given by
\begin{equation}\label{curv}
    -\lambda\la H_\Sigma,X\ra={\lambda} \int_{\theta_0}^\alpha
    \la X(\theta),\nu(\theta)\ra d\theta=\lambda \int_{\theta_0}^\alpha
    |X(\theta)|\,d\theta.
\end{equation}
Using~(\ref{polarinside}), we get
\[ \begin{array}{l}\vspace{6pt}\displaystyle
   \int_T \left\la
X(\pi^\Sigma(x)),\frac{\pi^\Sigma(x)-x}{|\pi^\Sigma(x)-x|}\right\ra
    \,dx \displaystyle \leq  \int_W \left\la
X(\pi^\Sigma(x)),\frac{\pi^\Sigma(x)-x}{|\pi^\Sigma(x)-x|}\right\ra
    \,dx \\
    \qquad\qquad\qquad\qquad\qquad\displaystyle=\int_{\theta_0}^\alpha \frac{|X(\theta)|}2 \left(R_1^2-\frac{R_1^2\sin^2\beta}{\sin^2(\beta+\alpha-\theta)}\right)d
    \theta,
\end{array} \]
and again for $\theta\to \alpha$,  with $\beta$ fixed, we have
\[
R_1^2-\frac{R_1^2\sin^2\beta}{\sin^2(\beta+\alpha-\theta)}=o(1),
\]
and hence
\begin{equation}\label{interior}
\int_T \left\la
X(\pi^\Sigma(x)),\frac{\pi^\Sigma(x)-x}{|\pi^\Sigma(x)-x|}\right\ra
    \,dx\leq o\left(\int_{\theta_0}^\alpha |X(\theta)| d\theta\right).
\end{equation}
Notice that
 $C$ contains a region
formed by $v'$, the ray $C_1 O$, the arc $\widetilde{MO}$ and some
arc concentric to $\widetilde{MO}$ but of bigger radius. We
express the subset of the boundary of $\Om$ bounding $C$ in polar
coordinates $(\rho,\theta)$ with respect to $C_1$ and $v'$ by the
equation $\rho=b_1+R_1+g_1(\theta)$, where $g_1(\theta)\to 0$ as
$\theta\to \alpha$, and $b_1$ is the distance between $O$ and the
intersection between $\partial\Om$ and the ray $C_1 O$, which we
denote by $S$. Then
\begin{equation}\label{exterior}
\begin{array}{l}\vspace{6pt}\displaystyle
   \int_C \Big\la
X(\pi^\Sigma(x)), \displaystyle\frac{\pi^\Sigma(x)-x}{|\pi^\Sigma(x)-x|}\Big\ra
    \,dx =  -\int_{\theta_0}^\alpha\int_{R_1}^{R_1+b_1+g_1(\theta)}|X(\theta)|\rho
    d\rho d\theta\\\vspace{6pt}
    \qquad\displaystyle=-\frac12\int_{\theta_0}^\alpha
    |X(\theta)|\left(2R_1b_1+b_1^2+g_1(\theta)^2+2(R_1+b_1)g_1(\theta)\right)d\theta\\\vspace{6pt}
    \qquad\displaystyle=
    -\frac12 (2R_1b_1+b_1^2)\int_{\theta_0}^\alpha |X(\theta)|d\theta\\
    \qquad\displaystyle\qquad -\frac12\int_{\theta_0}^\alpha
    |X(\theta)|\left(g_1(\theta)^2+2(R_1+b_1)g_1(\theta)\right)
    d\theta.
\end{array}
\end{equation}
Suppose now that the Euler equation~(\ref{main}) holds, that is,
\begin{equation}\label{eq_eul11}
    \begin{array}{l}\vspace{6pt}\displaystyle
\int_T \left\la
X(\pi^\Sigma(x)),\frac{\pi^\Sigma(x)-x}{|\pi^\Sigma(x)-x|}\right\ra
    \,dx  \displaystyle \\
    \qquad\qquad\qquad\qquad + \displaystyle \int_C \left\la
X(\pi^\Sigma(x)),\frac{\pi^\Sigma(x)-x}{|\pi^\Sigma(x)-x|}\right\ra
    \,dx  
=\lambda \la H_\Sigma, X \ra.
\end{array}
\end{equation}
Combining~(\ref{curv}),~(\ref{interior}) and~(\ref{exterior}) in
the above relationship, by comparison of the first order terms
with respect to $\int_{\theta_0}^\alpha
    |X(\theta)|\,d\theta$ as $\theta_0\to \alpha$, we obtain that
\begin{equation}\label{newb}
    b_1=\sqrt{R_1^2+{2\lambda}}-R_1.
\end{equation}
Moreover, by this choice of $b_1$ we have
\[
    -\frac12 (2R_1b_1+b_1^2)\int_{\theta_0}^\alpha
    |X(\theta)|d\theta = \lambda \la H_\Sigma, X \ra.
\]

From~(\ref{eq_eul11}) and~(\ref{exterior}), we conclude that
\[ \begin{array}{rl}\vspace{6pt}\displaystyle
& \displaystyle   \int_T \left\la
X(\pi^\Sigma(x)),\frac{\pi^\Sigma(x)-x}{|\pi^\Sigma(x)-x|}\right\ra
    \,dx \\
   &\displaystyle\qquad\qquad  -\frac12\int_{\theta_0}^\alpha
    |X(\theta)|\left(g_1(\theta)^2+2(R_1+b_1)g_1(\theta)\right)
    d\theta=0;
\end{array} \]
but since
\[
\int_T \left\la
X(\pi^\Sigma(x)),\frac{\pi^\Sigma(x)-x}{|\pi^\Sigma(x)-x|}\right\ra
    \,dx  >0,
\]
it follows
\[
\int_{\theta_0}^\alpha
    |X(\theta)|\left(g_1(\theta)^2+2(R_1+b_1)g_1(\theta)\right)
    d\theta>0.
\]
Minding that $X$ has an arbitrary support in $[\theta_0,\alpha]$,
this means that
\[
g_1(\theta)^2+2(R_1+b_1)g_1(\theta)>0
\]
and implies $g_1(\theta)>0$ for all $\theta\in [\theta_0,\alpha]$ whenever $\alpha$ is small enough (since otherwise $g_1(\theta)<-R_1-b_1$ which contradicts the fact that $g$ should be vanishing as $\theta\to\alpha)$. Hence, the part of $\partial\Om$ corresponding to the angular coordinate $\theta\in[\theta_0,\alpha]$ is, for small $\alpha$, more distant from $C_1$ than the arc $\sigma$ of the circumference with center $C_1$ passing through $S$, thus satisfying the equation $\rho(\theta)=R_1+b_1$. Thanks to convexity of $\Om$ we have then that any ray starting in $S$, directed inside the cone of points
with projection to $\Sigma$ in $O$, and belonging to a support
line to $\partial\Om$ in $S$, forms an angle not greater than $\pi/2$ with
the segment $\overline{SO}$ (mind that the angle of $\pi/2$
corresponds to the case when the ray is tangent to $\sigma$). As a
consequence, the part of $\Omega$ which lies in the angle (of
value $\varphi$) bounded by $u$ and the ray $OS$, is contained in the triangle $V_1$, formed by $u$, $\overline{OS}$ and the tangent in $S$ to $\sigma$.

Now fix a new vector field  $\hat{X}$, compactly supported in a small neighborhood of $O$ and such that $\hat{X}(O)$ is directed along $u$.
One has
\[
H_\Sigma(O)=\delta_O (\tau_Q+\tau_P),
\]
where $\tau_Q$ and $\tau_P$ are the unit vectors tangent in $O$ to
the arcs $\widetilde{PO}$ and $\widetilde{QO}$ respectively and
directed towards $P$ and $Q$ respectively. Since
\[
\la \hat{X}, \delta_O \tau_Q\ra=\la \hat{X}, \delta_O \tau_P\ra
=-|\hat{X}(O)|
  \sin\varphi,
  \]
  we get
\begin{equation}\label{curvcurv}
  -\lambda\la \hat{X}, H_\Sigma(O) 
  \ra=2\lambda |\hat{X}(O)|
  \sin\varphi.
\end{equation}
Now compute the contribution given by triangle $V_1$ to the first
term of the Euler equation~(\ref{main}). For this purpose we use
polar coordinates with respect to $O$ and the ray $OS$, with
$\theta\in[0,\varphi]$. It is clear that
\begin{equation}\label{triangle}
    V=\left\{(\rho,\theta)\colon 0\leq\theta\leq
    \varphi,0<\rho\leq\frac{b}{\cos\theta}\right\}.
\end{equation}
Therefore,
\begin{equation}\label{exteriortriangle}
    \begin{array}{l}\vspace{6pt}\displaystyle
   \int_{V_1} \left\la
\hat{X}(\pi^\Sigma(x)),\frac{\pi^\Sigma(x)-x}{|\pi^\Sigma(x)-x|}\right\ra
    \,dx \\
    \qquad\qquad\qquad\qquad\displaystyle
    =  -|\hat{X}(O)|\int_0^\varphi \cos(\varphi-\theta)\int_0^
    {\frac{b_1}{\cos\theta}}\rho d\rho d\theta\\\vspace{6pt}
    \qquad\qquad\qquad\qquad\displaystyle=-\frac{1}2 b_1^2 |\hat{X}(O)|
    \int_0^\varphi \frac{\cos(\varphi-\theta)}{\cos^2 \theta}d\theta\\
    \qquad\qquad\qquad\qquad\displaystyle=-\frac12 b_1^2
 |\hat{X}(O)| h(\varphi)\sin\varphi.
\end{array}
\end{equation}
    Reasoning in the same way with arc $\widetilde{PO}$ instead of the arc
    $\widetilde{PQ}$, we obtain the analogous triangle $V_2$ with
\begin{equation}\label{exteriortriangle1}
    \begin{array}{rl}\displaystyle
   \int_{V_2} \left\la
\hat{X}(\pi^\Sigma(x)),\frac{\pi^\Sigma(x)-x}{|\pi^\Sigma(x)-x|}\right\ra
    \,dx &\displaystyle=  -\frac12 b_2^2
 |\hat{X}(O)| h(\varphi)\sin\varphi,
\end{array}
\end{equation}
with
    $b_2\defeq \sqrt{R_2^2+{2\lambda}}-R_2$.
But since $\Om$ is convex, and one of the sides of $V_1$ (resp.\
$V_2$) is in the support line to $\Om$, we have
\begin{equation}\label{eq_estproj0}
    (\pi^{\Sigma})^{-1}(O)\cap \Om \subset V_1\cup V_2.
\end{equation}
 Let us write the Euler equation~(\ref{main}) with respect to the vector field
   $\hat{X}$. Letting $\Gamma:=(\Sigma\backslash\{O\})\cap
   \mbox{supp}\,{\hat{X}}$,
   thanks to~(\ref{curvcurv})
we get
\[
\begin{array}{rl}\displaystyle
 \int_{(\pi^{\Sigma})^{-1}(O)} \left\la
\hat{X}(\pi^\Sigma(x)),\frac{\pi^\Sigma(x)-x}{|\pi^\Sigma(x)-x|}\right\ra
    \,dx + 2\lambda |\hat{X}(O)|
  \sin\varphi + c_\Gamma =0,
\end{array}
\]
where by $c_\Gamma$ we denoted the sum of all the terms in the
Euler equation which involve integrals over $\Gamma$. Minding the
strict inclusion~(\ref{eq_estproj0}), and
using~(\ref{exteriortriangle}) and~(\ref{exteriortriangle1}), we
obtain
\begin{equation}\label{finaleuler}
\begin{array}{rl}\displaystyle
-\frac12 (b_1^2+b_2^2)
 |\hat{X}(O)| h(\varphi)\sin \varphi + 2\lambda |\hat{X}(O)|
  \sin\varphi + c_\Gamma < 0.
\end{array}
\end{equation}
 Since  $c_\Gamma$ contains only
integral terms, we have that $c_\Gamma$ can be made arbitrarily
small by choosing a sufficiently small support of $\hat X$, and
hence~(\ref{finaleuler}) may be satisfied, only if
    \begin{equation}\label{inequalityproof}
    \frac{4\lambda}{
    b_1^2+b_2^2}< h(\varphi),
    \end{equation}
    or, in other words, when $h(\varphi)$ is as in the
    the statement being proven, then the Euler equation is not
    satisfied. Finally, to prove the second claim, it remains
    to observe that $4\lambda/(
    b_1^2+b_2^2)> 1$, and hence with $h(\varphi) \leq 1$ the
    respective curved corner is not stationary.
\end{proof}

\begin{figure}
\includegraphics[height=10cm]{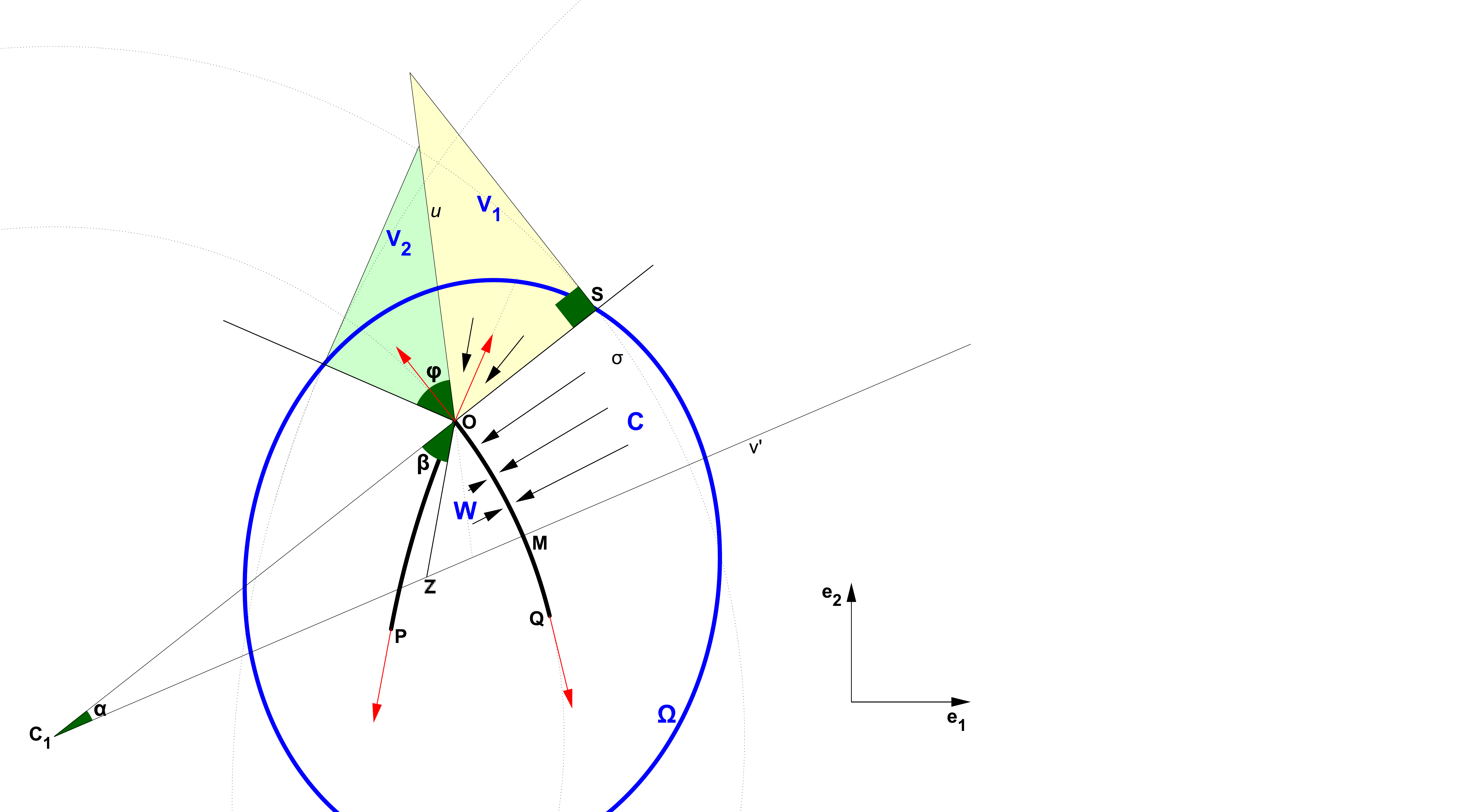}\caption{Construction of the proof of Proposition~\ref{prop_43}}\label{Figura 5}
\end{figure}



\section{The compliance case}\label{scompliance}

In this section we consider the case of a functional arising from the theory of elliptic equations:
\begin{equation}\label{ellpb}
\F(\Sigma)\defeq \int_\Omega
u_\Sigma(x)f(x)\,dx+\lambda\H(\Sigma).
\end{equation}
Here $\Omega\subset\R^2$ is a given bounded open subset, $f$ is a
given function, and $u_\Sigma$ is the unique solution of the PDE
\[
\left\{\begin{array}{ll}
-\Delta u=f& \hbox{ in }\Omega\setminus\Sigma,\\
u=0& \hbox{ on }\partial\Omega\cup\Sigma.
\end{array}\right.
\]
An integration by parts in the PDE above gives that the {\it
compliance} term $\int_\Omega u_\Sigma f\,dx$ appearing in the
functional $\F$ can be expressed in an equivalent way:
\[
\int_\Omega
u_\Sigma(x)f(x)\,dx=\max\Big\{\int_\Omega\big(2f(x)u-|\nabla
u|^2\big)\,dx\ :\ u\in W^{1,2}_0(\Omega\setminus\Sigma)\Big\}.
\]
 For
simplicity we assume that $f\in W^{1,2}(\R^2)$ and that $\Omega$ has a Lipschitz boundary. In fact, we could also consider the
case of a $p$-Laplace operator, and the similarity with the
average distance functional consists in the fact (shown
in Buttazzo, Santambrogio, 2007) that as $p\to+\infty$ the $p$-compliance
problem converges to the one with the average distance functional.
Here we limit ourselves to the case $p=2$. Also for simplicity we
have taken the Dirichlet condition $u=0$ on $\partial\Omega$; all
the arguments can be repeated for the Neumann case $\frac{\partial
u}{\partial\nu}=0$ on $\partial\Omega$.

The existence of a solution to the minimum problem
$$\min\big\{\F(\Sigma)\ :\ \Sigma\hbox{ closed connected subset of }\Omega\big\}$$
follows by an application of the $\check{\mbox{S}}$ver\'ak
compactness theorem (see Buttazzo, Santambrogio, 2007). Here we are
interested, as before, in the first order necessary conditions of
optimality.

Following Theorem 5.3.2 of Henrot, Pierre (2005), if $\phi_\eps$ is a one
parameter group of diffeomorphisms satisfying~(\ref{generator}),
setting $\Sigma_\eps\defeq \phi_\eps(\Sigma)$, $u\defeq u_\Sigma$
and $u_\eps=u_{\Sigma_\eps}$, we have as $\eps\to 0$ that
$\frac{u_\eps-u}{\eps}\to u'$ in $L^2(\Omega)$, where $u'$ satisfies
the PDE
\[
\left\{\begin{array}{ll}\displaystyle
-\Delta u'=0\hbox{ in }\Omega\setminus\Sigma,\\\displaystyle
u'=0\hbox{ on }\partial\Omega,\ u'=-\nabla u\cdot X\hbox{ on }\Sigma.
\end{array}\right.
\]
Note that the boundary conditions in the above equation are
understood in the weak sense, i.e.\ $u'+\nabla u\cdot X\in
W_0^{1,2}(\R^2)$.
 Therefore, the first variation
argument applied to the functional $\F$ gives
\[
\frac{\partial}{\partial\eps}
\F(\Sigma_\eps){\Big|}_{\eps=0}=\int_\Omega u'f\,dx-\lambda\la
H_\Sigma,X\ra.
\]
Suppose now that $\Omega=\Omega^+\cup \Omega^-$ with $\Sigma\subset
\partial\Omega^+\cap \partial\Omega^-$. Then, if $\Sigma$,
$\partial \Omega$ and $f$ provide sufficient regularity for $u$
and $u'$ so that the Green formula can be applied, we have
\[ \begin{array}{rl}\vspace{6pt}\displaystyle
    \int_{\Omega^+}
u'f\,dx  =- \int_{\Omega^+} u'\Delta u\,dx &\displaystyle= \int_{\Omega^+} \nabla
u'\nabla u\,dx - \int_{\partial \Omega^+} u'\frac{\partial
u}{\partial n}\, d\H \\\vspace{6pt}
&\displaystyle=\int_{\Omega^+} \nabla u'\nabla u\,dx + \int_\Sigma \nabla u
\cdot X \frac{\partial u}{\partial n}\, d\H \\
&\displaystyle\qquad\qquad - \int_{\partial \Omega^+ \setminus (\partial
\Omega\cup \Sigma)} u'\frac{\partial u}{\partial n}\, d\H,
\end{array} \]
where $n$ stands for the external normal to $\Omega^+$.
 But
\[ \begin{array}{rl}\vspace{6pt}\displaystyle
\int_{\Omega^+} \nabla u'\nabla u\,dx &\displaystyle =- \int_{\Omega^+} u
\Delta u'\,dx + \int_{\partial \Omega^+} u'\frac{\partial
u}{\partial n}\, d\H \\
& \displaystyle=
 -\int_{\partial \Omega^+ \setminus (\partial \Omega\cup
\Sigma)} u\frac{\partial u'}{\partial n}\, d\H.
\end{array} \]
Thus,
\begin{equation}\label{eq_iom+}
\begin{array}{rl}\vspace{6pt}\displaystyle
\int_{\Omega^+} u'f\,dx  =  & \displaystyle\int_\Sigma \nabla u^+ \cdot X
\frac{\partial u^+}{\partial n}\, d\H \\
& \displaystyle -
 \int_{\partial \Omega^+ \setminus (\partial \Omega\cup
\Sigma)} \left( u\frac{\partial u'}{\partial n}+u'\frac{\partial
u}{\partial n}\right)\, d\H,
\end{array} 
\end{equation}
where $\nabla u^+$ stands for the trace on $\Sigma$ of the gradient
of $u$ restricted to $\Omega^+$, and $\frac{\partial u^+}{\partial
n}$ stands for the trace of the respective normal derivative.
 Analogously, minding that the external normal to $\Omega^-$
over $\partial\Omega^+\cap \partial\Omega^-$ is given by $-n$, we
get
\begin{equation}\label{eq_iom-}
\begin{array}{rl}
\displaystyle \int_{\Omega^-} u'f\,dx  =  & \displaystyle -\int_\Sigma \nabla u^- \cdot X
\frac{\partial u^-}{\partial n}\, d\H \\
& \displaystyle +
 \int_{\partial \Omega^+ \setminus (\partial \Omega\cup
\Sigma)} \left( u\frac{\partial u'}{\partial n}+u'\frac{\partial
u}{\partial n}\right)\, d\H,
\end{array}
\end{equation}
where $\nabla u^-$ stands for the trace on $\Sigma$ of the
gradient of $u$ restricted to $\Omega^-$, and $\frac{\partial
u^-}{\partial n}$ stands for the trace of the respective normal
derivative. From~(\ref{eq_iom+}) and~(\ref{eq_iom-}) we obtain
\[
\int_{\Omega} u'f\,dx  =    \int_\Sigma \nabla u^+ \cdot X
\frac{\partial u^+}{\partial n}\, d\H - \int_\Sigma \nabla u^-
\cdot X \frac{\partial u^-}{\partial n}\, d\H.
\]
Recalling that
\[
\nabla u^\pm =\frac{\partial u^\pm}{\partial n} n,
\]
since the tangential derivatives of $u^\pm$ over $\Sigma$ vanish
(because $u^\pm=u=0$ on $\Sigma$), we get
\[
\int_{\Omega} u'f\,dx  =  \int_\Sigma  \left( \left(\frac{\partial
u^+}{\partial n}\right)^2 - \left(\frac{\partial u^-}{\partial
n}\right)^2 \right) X\cdot n \, d\H.
\]
Hence,
\[
\frac{\partial}{\partial\eps}\F(\Sigma_\eps){\Big|}_{\eps=0}=\int_\Sigma
\left( \left(\frac{\partial u^+}{\partial n}\right)^2 -
\left(\frac{\partial u^-}{\partial n}\right)^2 \right) X\cdot n \,
d\H-\lambda\la H_\Sigma,X\ra.
\]
Since this holds for every vector field $X$, we deduce the Euler
equation that must hold for every minimizer of $\mathcal{F}$:
\[
\left(\frac{\partial u^+}{\partial n}\right)^2 -
\left(\frac{\partial u^-}{\partial n}\right)^2 =\lambda\la
H_\Sigma, n\ra.
\]

\end{document}